\documentclass{article}
\usepackage{amsmath,amssymb,amsthm}
\usepackage{tikz}
\usepackage{color}
\usepackage[toc]{appendix}
\usepackage{graphicx}
\usepackage{fancyhdr}
\usepackage{enumitem}
\usepackage{bbm}
\usepackage{parskip}
\usepackage{float}
\usepackage{chngpage}
\usepackage{calc}
\usepackage{bigints}
\usepackage{array}
\usepackage{booktabs}
\usepackage{rotating}
\usepackage{multirow}
\usepackage{adjustbox}
\usepackage{tabularx}
\usepackage{verbatim}
\usepackage{mathtools}
\usepackage{ragged2e}
\usepackage[makeroom]{cancel}
\usepackage{caption}
\usepackage{hyperref}
\usepackage{caption}
\usepackage{subcaption}
\usepackage{appendix}
\usepackage{pgfplots}
\pgfplotsset{compat=1.16}
\usepackage[english]{babel}
\usepackage{hyphenat}
\usepackage[makeindex]{imakeidx}
\usepackage{amsmath}

\newcommand{\wh}{\widehat}
\newcommand{\ang}[1]{\langle #1\rangle}
\def\R{\mathbb{R}}
\def\C{\mathbb{C}}

\def\N{\mathbb{N}}

\def\E{\mathbb{E}}

\newcommand{\vertiii}[1]{{\left\vert\kern-0.25ex\left\vert\kern-0.25ex\left\vert #1 
    \right\vert\kern-0.25ex\right\vert\kern-0.25ex\right\vert}}

\usetikzlibrary{datavisualization}
\usetikzlibrary{matrix}
\usetikzlibrary{datavisualization.formats.functions}

\newtheorem{theorem}{Theorem}[section]
\newtheorem{proposition}[theorem]{Proposition}
\newtheorem{corollary}[theorem]{Corollary}
\newtheorem{lemma}[theorem]{Lemma}
\newtheorem{definition}[theorem]{Definition}

\newtheorem{remark}[theorem]{Remark}
\newtheorem*{notation*}{Notation}

\newcount\madechanges
\def\caution#1{\ifnum \madechanges=1 \affixmessage{#1}%
\else \relax \fi}
\def\affixmessage#1{\marginpar{{\footnotesize  \em #1} \openup
    -.3\baselineskip }}
\makeatletter
\def\section{\@startsection {section}{1}{\z@}{3.25ex plus 1ex minus
		.2ex}{1.5ex plus .2ex}{\large\bf}}
\def\subsection{\@startsection{subsection}{2}{\z@}{3.25ex plus 1ex minus
		.2ex}{1.5ex plus .2ex}{\normalsize\bf}}
\@addtoreset{equation}{section} 
\makeatother
\title{Regularization by noise for Gevrey well-posedeness of a weakly
  hyperbolic operator}

\author{ Enrico Bernardi \thanks{Dipartimento di Scienze Statistiche Paolo Fortunati, Università di Bologna, Bologna, Italy. \textbf{e-mail}: enrico.bernardi@unibo.it} \and
Alberto Lanconelli\thanks{Dipartimento di Scienze Statistiche Paolo
  Fortunati, Università di Bologna, Bologna, Italy. \textbf{e-mail}:
  alberto.lanconelli2@unibo.it}
}
\date{\today}
\begin{document}
\maketitle
\bigskip
\begin{abstract}
We present an example of a linear partial differential equation whose Cauchy problem becomes well-posed when perturbed by noise. Specifically, we make clear how  a suitable multiplicative Stratonovich perturbation of Brownian type renders a weakly hyperbolic operator with double involutive characteristics well-posed in the $C^{\infty}$-category, while its deterministic counterpart is only well-posed in the Gevrey $ s $ classes with $ 1 \leq s <2 $ . 
\end{abstract}


\allowdisplaybreaks
	
\section{Introduction and statement of the main result}	

It is  well-known that the Cauchy problem
\begin{align}
  \label{eq:CP}
  \begin{cases}
    \mathrm{P}(t,x,\partial_{t},\partial_{x})u(t,x)=0,& (t,x)\in[0,T]\times \R^{n};\\
    \partial_{t}^{j}u(0,x)=u_{j}(x),& x\in\R^{n},j=0, \ldots, m-1,
  \end{cases}
\end{align}
where $\mathrm{P}$ is a linear hyperbolic differential operator of order $ m\in\mathbb{N}
$ with analytic coefficients, may fail to be
well-posed in the $C^{\infty}$-category if $\mathrm{P}$ is weakly
hyperbolic, that is if some of its characteristic roots coincide. When
this happens, and $ r\in\mathbb{N}$ denotes the largest multiplicity of these roots,
then a classical result in \cite{BroGevrey} shows that for arbitrary
lower order terms of the operator $\mathrm{P}$, one can in general only
expect well-posedness in the Gevrey classes $\gamma^{(s)}$, with $ 1
\leq s < \frac{r}{r-1} $. Moreover, the threshold $ \frac{r}{r-1} $ is usually optimal.
We recall that  $f\in\gamma^{(s)}\left(\R^{n}\right)$, the Gevrey class of order $s\geq 1$, if for any compact set $K \subset \R^{n}$ there exist constants $C>0$ and $h>0$ such that
\begin{align*}
\left|\partial_{x}^{\alpha} f(x)\right| \leq C h^{-|\alpha|}(\alpha!)^{s}\quad\mbox{ for all }x \in K\mbox{ and }\alpha \in \N^{n}.
\end{align*}
We also denote $ \gamma^{(s)}_{0}(\R^n) = \gamma^{(s)}(\R^n) \cap C_{0}^{\infty}(\R^n)$.

In this paper we study an explicit example of a weakly hyperbolic operator with a manifold of involutive double characteristics $ \Sigma_{2}= \{t \in [0,T], x \in \R
: (t,x,\tau=0;\xi\neq 0)\} $ and  $m=r=2$, i.e. $T_{\rho}\Sigma_{2}^{\sigma} \subset T_{\rho}\Sigma_{2} $ the
symplectic dual of the tangent space at $ \rho \in \Sigma_{2} $ is
contained in $ T_{\rho}\Sigma_{2}$. The operator reads as
\begin{align}
  \label{eq:MOD}
  \begin{cases}
  (\partial_{t}^{2} +
    i\partial_{x})u(t,x)=0, & t\in [0,T], x\in\mathbb{R};\\
    u(0,x)=\varphi_0(x), \partial_tu(0,x) = \varphi_1(x), & x\in\mathbb{R}.
    \end{cases}
\end{align}
Albeit (\ref{eq:MOD}) is admittedly very simple, its principal part  is nonetheless one of the only three possible cases in which a hyperbolic quadratic form can be invariantly
represented in suitable symplectic coordinates. Thus, (\ref{eq:MOD}) is actually the microlocal model for \emph{any} hyperbolic operator with double involutive characteristics. Many physical phenomena are driven by equations of the type (\ref{eq:MOD}): conical refraction of light rays through prisms, see e.g \cite{MR544247} and \cite{MR637495}, dynamics of wet-dry regimes in linearized Saint-Venant shallow
waters systems see e.g. \cite{MR1339932} and \cite{MR2761079}, to name just a few.\\
To be more specific, localizing the principal symbol at a double
point, we end up with a symplectic quadratic form  of one of  these three
invariant cases, as proved in \cite{MR781536} and \cite{MR1178557}:
\begin{align}
  \label{eq:QFS}
    &\bullet\!\!\!\!\quad p(t,\tau,x;\xi)= 2\lambda t \tau , ~\lambda >0;\nonumber\\
    &\bullet\!\!\!\!\quad p(t,\tau,x;\xi)= -\tau^{2};\\
    &\bullet\!\!\!\!\quad p(t,\tau,x;\xi)= -\tau^{2} +2\tau\xi_{1} + x_{1}^{2}\nonumber.
\end{align}
The first case in (\ref{eq:QFS}) corresponds to the so called
effectively hyperbolic case, while the third one is the only invariant
case where its Hamiltonian matrix has a Jordan block of size $ 4
$. Stochastic generalizations of the first item in (\ref{eq:QFS}) were
recently studied in \cite{TIN}, while some results for associated stochastic partial differential equations in this last case where also obtained in \cite{SDC}. For the analysis of other weakly hyperbolic partial differential equations perturbed by noise we refer the reader to \cite{MR4087470} and \cite{MR4219195}. \\
It is immediate to check that the lower order term in (\ref{eq:MOD}) does not satisfy
the necessary Ivrii-Petkov-H\"ormander condition, which for this model
states that the sub-principal symbol of $\mathrm{P}$ must vanish on $\Sigma_{2}$, see \cite{MR492751} and \cite{MR427843}. This implies that the Gevrey threshold $ s=2 $ is
the natural barrier for equation (\ref{eq:MOD}).

Aim of the present manuscript is to understand the effect of noise on this particular setup. Is it possible to gain in-mean well-posedness for a suitable stochastic perturbation of (\ref{eq:MOD})? The answer presented and proved here is affirmative,
positioning this outcome among the already vast array of results capped under the general
title of regularization by noise.\\
Dating back to results of  Khasminskii \cite{KHAS}, where examples of unstable systems of
ordinary differential equations become asymptotically stable when their parameters are
perturbed by white noise, passing through the works of Zvonkin \cite{Zvonkin} and Veretennikov \cite{Veretennikov}, that prove existence of a unique strong solution for stochastic differential equation with merely bounded and measurable drift, and reaching for instance to the theory of transport-type stochastic perturbations and the foundational work of Flandoli–Gubinelli–Priola \cite{FGP},
where we can see that, in those setting, noise restores uniqueness for
transport equations below deterministic thresholds, we do know that there are many
instances of noise--regularized behaviors. 
See also the survey \cite{10.1007/978-3-319-74929-7_3} and lecture notes \cite{Flandoli2011} for a broader overview of the phenomenon.\\
The general philosophy of regularization by noise is that stochastic
perturbations may  introduce effective coercivity or averaging mechanisms that are absent in the deterministic problem.
In many examples, this becomes visible after converting the equation
to It\^o form and  exploiting quadratic variation
\cite{FGP}. Something along these lines is happening here as well.\\
On a slightly different angle but in a similar vein too, we may also cite some recent results \cite{lü2026uniquecontinuationpropertystochastic} on
the unique continuation property for stochastic wave equations, which  seem to
present a similar scheme: stochastic Carleman estimates restore
unicity properties near characteristic surfaces, which are known to
usually fail in the deterministic case. 

We may now state our main result whose proof is postponed to Section \ref{SEE} below. We first rewrite (\ref{eq:MOD}) as
\begin{align}
\label{eq:STRA 1}
\begin{cases}
\partial_tu(t,x) = v(t,x),& t\in(0,T], x\in\mathbb{R};\\
\partial_tv(t,x) = -i\partial_xu(t,x), & t\in(0,T], x\in\mathbb{R};\\
u(0,x)=\varphi_0(x),v(0,x)=\varphi_1(x),& x\in\mathbb{R};
\end{cases}
\end{align}
then, we consider the stochastic partial differential equation
\begin{align}
  \label{eq:STRA}
  \begin{cases}
    dU(t,x) = V(t,x) dt+ \sigma \partial_xU(t,x)\circ dB(t),& t\in(0,T], x\in\mathbb{R}\\
    dV(t,x) = -i\partial_x U(t,x)dt & t\in(0,T], x\in\mathbb{R}\\
    U(0,x)=\varphi_0(x),V(0,x)=\varphi_1(x),& x\in\mathbb{R},
  \end{cases}
\end{align}
where $ \sigma$ is a positive real number and $\{B(t)\}_{t\geq 0} $ is a standard one dimensional Brownian motion. As usual by  $ \circ $ we  denote
Stratonovich-type integration, see for instance \cite{MR1214374}. Equation \eqref{eq:STRA} represents a stochastic perturbation of (\ref{eq:STRA 1}). The next theorem states the better well-posedness of (\ref{eq:STRA}) over (\ref{eq:STRA 1}); in the sequel we set $H_{x}^{\infty}:=
\bigcap_{k\ge0}H_{x}^{k}$ and $H_{x}^{k}$ denoting the usual Sobolev space
in the $x$ variable.
\begin{theorem}
  \label{MTH}
The Cauchy problem for the stochastic Stratonovich perturbation (\ref{eq:STRA}) is well-posed in $H_x^\infty$ in the continuous 
sense, i.e. it has a unique solution process in
\begin{align*}
U\in \bigcap_{k\ge0}C\big([0,T];L^2(\Omega;H_x^k)\big) ,
\qquad
V\in \bigcap_{k\ge0} C\big([0,T];L^2(\Omega;H_x^{k-\frac12})\big)~.
\end{align*}
\end{theorem}

\begin{remark}
This result can be compared with the known result, a proof of which is sketched  in the next section, which states that for \(s>2\) the Cauchy problem for \eqref{eq:STRA 1} is not locally solvable in \(\gamma^{(s)}\); in particular, the Cauchy problem for \eqref{eq:STRA 1} is not \(C^{\infty}\)-solvable.
\end{remark}

The paper is organized as follows: in Section \ref{DR} we review the
known facts that fix the Gevrey index $ s=2 $ as the natural threshold
for the deterministic model \eqref{eq:STRA 1}. 
In Section \ref{SEE} we prove the main result, which is direct
consequence of Theorem \ref{thm:Global} below and of the
stochastic energy estimates of Lemma \ref{lem:sobolevMoment2}.
\section{A brief review of the deterministic case}
\label{DR}
The next two subsections present  simple self-contained arguments showing how
the threshold Gevrey $ s=2 $ appears naturally for this
model. Needless to say every result given here is well-known and
belongs to a historically  well-established theoretical framework. An exhaustive and
complete collection of results can be found for instance in
\cite{NTLEC}, source of some of the theorems adapted and used below. 

\subsection{Sufficiency}\label{sec:suff}

We present a very short and sketchy review of how to get Gevrey energy estimates
for the model (\ref{eq:MOD}). We consider here, with Fourier
derivatives $ D_{y}:= \frac{1}{i}\partial_{y} $, the slightly more
general symbol
\begin{align*}
  \mathrm{P}(x,D) = -D_{t}^{2} + AD_{x}
\end{align*}
with $ A\in \C $. Taking the Fourier transform in $x$ and setting 
\begin{align*}
\wh u(\xi):= \int_{\R}e^{-ix\xi}u(x)dx,\quad \xi\in\mathbb{R}, 
\end{align*}
we get
\begin{align*}
  2i\Im \langle \mathrm{P}\wh u,-D_{t} \wh u \rangle = D_{t}|D_{t}\wh u|^{2} + 2i
  \Im \xi \langle A \wh u, D_{t}\wh u \rangle
\end{align*}
and
\begin{align}
  \label{eq:DER2}
  2i\Im \langle D_{t}\wh u, \wh u \rangle = D_{t}|\wh u|^{2}.
\end{align}
With the proviso, not affecting the gist of the argument,  that $ \xi $ is positive and large and with $ \lambda >0 $ to be suitably chosen, define the Gevrey norm
\begin{align*}
  \vertiii{v}_{s}^{2}(\xi):= \int_{0}^{\infty}e^{2\lambda t \xi^{1/s}}|v(t;\xi)|^{2}dt,\quad \xi\in\mathbb{R},
\end{align*}
and inner product
\begin{align*}
	\displaystyle \langle f, g
	\rangle_{s}(\xi) = \int_{0}^{\infty}e^{2\lambda t
		\xi^{1/s}}f(t;\xi)\bar{g}(t;\xi)dt ,\quad \xi\in\mathbb{R}.
\end{align*}
Then, multiplying the two previous identities by $ \displaystyle e^{2\lambda t
  \xi^{1/s}} $ and integrating by parts we get 
\begin{align}
  \label{eq:G2EST}
  2\Im\langle P \wh u ,-D_{t} \wh u \rangle(\xi) \geq 2 \lambda
  \xi^{1/s}\vertiii{D_{t}\wh u}_{s}^{2}(\xi) + 2
  \Im \xi \langle A \wh u, D_{t}\wh u \rangle(\xi).
\end{align}
Since we also have from (\ref{eq:DER2}) that
\begin{align*}
  \vertiii{D_{t}\wh u}_{s}^{2}(\xi) \geq \lambda^{2}\xi^{2/s}\vertiii{\wh u}_{s}^{2}(\xi),
\end{align*}
using Cauchy-Schwarz inequality and  
\begin{align*}
\displaystyle 2 -
\frac{1}{s} < \frac{3}{s} 
\end{align*}
we see that we can estimate the last
summand in (\ref{eq:G2EST}) with a choice of $ \lambda $ and therefore, after an inverse Fourier
transform,  we have reached a Gevrey $ s $
with $ s <2  $ energy estimate. Standard arguments allow us then to
conclude.

\begin{remark}
For the sake of completeness we can also remark that the
Cauchy problem for (\ref{eq:MOD}) fails when $ s=2 $, albeit in a
slightly weaker sense. Since the main focus here is however on the lack of
solvability for $ s>2 $, we point the interested reader to some
classical literature for these cases: \cite{NTLEC}.
\end{remark}

\subsection{Sharpness}\label{sec:sharp}

Now we want to verify that if $ s>2 $ the Cauchy problem for
(\ref{eq:MOD}) is ill-posed. This is done in Theorem \ref{ill1} below.
We follow the arguments of \cite{NTLEC} Section 6.3, starting with the following important definition.

\begin{definition}
  We say that the Cauchy problem for the second order differential
  operator $\mathrm{P}$ is locally solvable in \(\gamma^{(s)}\) at the origin
  if for any \(\Phi=\left(\varphi_0, \varphi_1\right) \in
  \left(\gamma^{(s)}\left(\R^{n}\right)\right)^{2}\), there
  exists  a neighborhood \(U_{\Phi}\) of the origin such that there exists $u \in C^{\infty}\left(U_{\Phi}\right)$ satisfying
\begin{align*}
\begin{cases}
\mathrm{P}u=0,& \mbox{ in }U_{\Phi}; \\
D_{0}^{j} u\left(0, x^{\prime}\right)=u_{j}\left(x^{\prime}\right), j=0,1, & \left(0, x^{\prime}\right) \in U_{\Phi} \cap\left\{x_{0}=0\right\}.
\end{cases}
\end{align*}
\end{definition}

Then our main result in this subsection reads as follows.

\begin{theorem}
  \label{ill1}
  If \(s>2\) then the Cauchy problem (\ref{eq:MOD}) is not
  locally solvable in \(\gamma^{(s)}\). In particular the Cauchy
  problem (\ref{eq:MOD}) is not \(C^{\infty}\) solvable.
\end{theorem}

To prove Theorem (\ref{ill1}) we recall the following result (see
e.g. \cite{NTLEC}, \cite{MR97628} or \cite{MR599580} Proposition 4.1).

\begin{proposition}
	\label{pro1}
	Let \(K\) be a compact set of \(\R^{n}\) and \(h>0\) be
	fixed. Assume that the Cauchy problem for \(\mathrm{P}\) is locally solvable
	in \(\gamma^{(s)}\) at the origin. Then there is \(\delta>0\) such
	that for any \(\left(\varphi_0\left(x^{\prime}\right),
	\varphi_1\left(x^{\prime}\right)\right) \in\left(\gamma^{(s),
		h}(K)\right)^{2}\) there exists a unique \(u(x) \in
	C^{2}\left(D_{\delta}\right)\)  verifying \(Pu=0\) in \(D_{\delta}\) and \(D_{0}^{j} u\left(0, x^{\prime}\right)=u_{j}\left(x^{\prime}\right)\) on \(D_{\delta} \cap\left\{x_{0}=0\right\}\).
\end{proposition}

We also need the following proposition.

\begin{proposition}\label{pro2}
	Assume that \(\theta \in C_{0}^{\infty}(\R)\) is an even
	function such that \(\theta \notin \gamma_{0}^{(2)}(\R)\). Let \(\Omega\) be a neighborhood of the origin of \(\R\)
	such that \(\operatorname{supp} \theta \subset \Omega
	\cap\left\{x_{0}=0\right\}\).
	Then the Cauchy problem
	\begin{align}
		\label{eq:MODN}
		\begin{cases}
			(\partial_{t}^{2} +
			i\partial_{x})u(t,x)=0,& t\in(0,T],x\in\mathbb{R};\\
			u(0,x)=0, V(0,x) = \theta(x),&  x\in\mathbb{R},
		\end{cases}
	\end{align}
	has no \(C^{2}(\Omega)\) solution.
\end{proposition}

\begin{proof}
	We define
	\begin{align*}
		U_{\rho}(t,x)= e^{-i\rho^{2}x + iz\rho(T-t)} 
	\end{align*}
	and we choose $ z =-i $ so that $ U_{\rho}(t,x) $ is a particular
	solution of the homogeneous equation.
	Suppose that (\ref{eq:MODN}) has a \(C^{2}(\Omega)\) solution. Applying Holmgren's Theorem ( see Proposition 6.3 in \cite{NTLEC})  we conclude that we can assume \(u(x)=0\) if \(\left|t\right| \leq T\) and \(\left|x\right| \geq r\) with some small \(T>0\) and \(r>0\).\\	
	We note that, with $ (u,v) = \int_{\R}u(x)\bar{v}(x)dx $,
	\[
	\begin{aligned}
		\int_{0}^{T}\left(\mathrm{P} U_{\rho}, u\right) dt= & \int_{0}^{T}\left(U_{\rho}, \mathrm{P} u\right) dt+i\left(D_{t} U_{\rho}(T), u(T)\right) \\
		& +i\left(U_{\rho}(T), D_{t} u(T)\right)
		-i\left(U_{\rho}(0), D_{t} u(0)\right)
	\end{aligned}
	\]
	because \(u(0)=0\). From this we have
	\begin{equation}
		\label{eq:IDE1}
		\begin{aligned}
			\left(D_{t} U_{\rho}(T), u(T)\right)+\left(U_{\rho}(T), D_{t} u(T)\right) =\left(U_{\rho}(0), D_{t} u(0)\right).
		\end{aligned}
	\end{equation}
	We see that the left-hand side on (\ref{eq:IDE1}) is \(O\left(\rho\right)\). On the other hand the right-hand side is	
	\begin{align*}
		e^{\rho T}\hat{\theta}(-\rho^{2})
	\end{align*}
	where \(\hat{\theta}\) is the Fourier transform of \(\theta\). This comparison yields that, since $ \theta$ is even 	
	\begin{align*}
		|\hat{\theta}(\rho)| \leq C \rho^{1/2} e^{-c \rho^{1 / 2}} \leq C^{\prime} e^{-c^{\prime} \rho^{1 / 2}}
	\end{align*}
	with some \(c^{\prime}>0\) and $ \rho $ large. We   deduce that \(\theta \in \gamma_{0}^{(2)}(\R)\) which is a contradiction.
\end{proof}

\begin{proof}[Proof of Theorem \ref{ill1}]
	
Suppose that the Cauchy problem  (\ref{eq:MOD}) is locally solvable in \(\gamma^{(s)}\) at
the origin with some \(s>2\).
Choose \(s^{\prime}\) so that \(s>s^{\prime}>2\),  and choose a compact
neighborhood \(K\) of the origin of \(\R\) and a positive
\(h>0\).
Then from Proposition \ref{pro1} there exists \(D_{\delta}\) such that
the Cauchy problem for \(\mathrm{P}\) has a \(C^{2}\left(D_{\delta}\right)\)
solution for any Cauchy data in \(\left(\gamma^{(s),
    h}(L)\right)^{2}\). We now choose
 \(\theta \in \gamma_{0}^{\left(s^{\prime}\right)}(\R)\) so
that \(\operatorname{supp}  \theta \subset K
\cap\left(D_{\delta} \cap\left\{x_{0}=0\right\}\right)\) which satisfy
the conditions in Proposition \ref{pro2} below. Since it is clear that \(
\theta \in \gamma_{0}^{(s), h}(K)\) because \(s>s^{\prime}\) one can
apply Proposition \ref{pro1} to conclude that there is a
\(C^{2}\left(D_{\delta}\right)\) solution to (6.16),  while this
contradicts  proposition \ref{pro2}.
\end{proof}








\section{The stochastic energy estimate}
\label{SEE}
In this section we prove the main theorem of our work, i.e. Theorem \eqref{MTH}. We start taking the Fourier transform with respect to $x$ in \eqref{eq:STRA} to get
\begin{align}
  \label{eq:MSF}
  \begin{cases}
    d\wh U(t;\xi) = \wh V(t;\xi)dt + i\sigma\xi\wh U(t;\xi)\circ dB(t),& t\in (0,T],\xi\in\mathbb{R};\\
d\wh V(t;\xi) = \xi\wh U(t;\xi)dt,& t\in (0,T],\xi\in\mathbb{R};\\
    \wh U(0;\xi)=\wh \varphi_0(\xi),\wh V(0;\xi)=\wh\varphi_1(\xi),& \xi\in\mathbb{R},
  \end{cases}
  \end{align}
which in It\^o's form reads
\begin{align}
\label{ito}
\begin{cases}
    d\wh U(t;\xi) = \left[\wh V(t;\xi)-\frac{\sigma^2}{2}\xi^2\wh U(t;\xi)\right]dt + i\sigma\xi\wh U(t;\xi)dB(t),& t\in (0,T],\xi\in\mathbb{R};\\
d\wh V(t;\xi) = \xi\wh U(t;\xi)dt,& t\in (0,T],\xi\in\mathbb{R};\\
\wh U(0;\xi)=\wh \varphi_0(\xi),\wh V(0;\xi)=\wh\varphi_1(\xi),& \xi\in\mathbb{R},
\end{cases}
\end{align}
For $t\in [0,T]$ and $\xi\in\mathbb{R}$ we define
\begin{align}\label{m}
m_1(t;\xi):=\E|\wh U (t;\xi)|^2,\qquad m_2(t;\xi):=\E\,\Re\left(\overline{\wh U (t;\xi)}\wh V (t;\xi)\right),\qquad m_3(t;\xi):=\E|\wh V(t;\xi)|^2.
\end{align}
The idea is to do an energy estimate, of the type of Section
(\ref{sec:suff}), at the expectation level. It is here that the role
of the perturbation $ \sigma $ can be seen to play its role.

\begin{lemma}[Second moments system]\label{lem:momentDerivation}
For any $\xi\in\mathbb{R}$ the function 
\begin{align*}
t\mapsto m(t;\xi)=(m_1(t;\xi),m_2(t;\xi),m_3(t;\xi))^{\top} 
\end{align*}
defined in \eqref{m} solves the linear system of ordinary differential equations
\begin{align}\label{ODE}
\begin{cases}
\dot{m_1}(t;\xi) = 2m_2(t;\xi),\\
\dot{m_2}(t;\xi) = \xi m_1(t;\xi) - \frac{\sigma^2}{2}\xi^2 m_2(t;\xi)+m_3(t;\xi),\\
\dot{m_3}(t;\xi) = 2\xi m_2(t;\xi);
\end{cases}
\end{align}
here, $\dot{m_i}(t;\xi):=\frac{d}{dt}m_i(t;\xi)$ for $i=1,2,3$. 
\end{lemma}

\begin{proof}
Applying It\^o's formula to $|\wh U(t;\xi)|^2=\wh U(t;\xi)\overline{\wh U(t;\xi)}$ gives
\begin{align*}
d|\wh U(t;\xi)|^2
= \overline{\wh U(t;\xi)}d\wh U(t;\xi) + \wh U(t;\xi) d\overline{\wh U(t;\xi)} + d\wh U(t;\xi)\overline{\wh U(t;\xi)}.
\end{align*}
Now, if we substitute $d\wh U(t;\xi)$ and $d\overline{\wh U(t;\xi)}$ with the corresponding expressions from \eqref{ito}, we find
\begin{align*}
d|\wh U(t;\xi)|^2
= 2\Re(\overline{\wh U(t;\xi)}\wh V(t;\xi))dt.
\end{align*}
Taking expectations yields $\dot{m_1}(t;\xi) = 2m_2(t;\xi)$, i.e.  the first equation in \eqref{ODE}.\\
Moreover, using It\^o's product rule we get
\begin{align*}
d(\overline{\wh U(t;\xi)}\wh V(t;\xi))=d\overline{\wh U(t;\xi)}\wh V(t;\xi)+\overline{\wh U(t;\xi)}d\wh V(t;\xi)+d\overline{\wh U(t;\xi)}d\wh V(t;\xi).
\end{align*}
Since $d\wh V(t;\xi)$ has no Brownian part, the last term vanishes and, upon taking real parts and expectations, we obtain
\begin{align*}
\dot{m_2}(t;\xi)&=\E|\wh V(t;\xi)|^2 + \xi\E|\wh U(t;\xi)|^2 - \frac{\sigma^2}{2}\xi^2\E\Re(\overline{\wh U(t;\xi)}\wh V(t;\xi))\\
&= \xi m_1(t;\xi) - \frac{\sigma^2}{2}\xi^2 m_2(t;\xi)+m_3(t;\xi).
\end{align*}
This corresponds to the second line in \eqref{ODE}.\\
Lastly, applying It\^o's formula to $|\wh V(t;\xi)|^2$ and recalling that $dV(t;\xi)$ has no stochastic part, we see that
\begin{align*}
d|\wh V(t;\xi)|^2 = 2\xi\,\Re(\overline{\wh U(t;\xi)}\wh V(t;\xi))dt.
\end{align*}
Taking expectations yields $\dot{m_3}(t;\xi) = 2\xi m_2(t;\xi)$ thus completing the proof.
\end{proof}

\quad

\begin{itemize}
\item \textbf{Eigenvector decomposition}
\end{itemize}

The eigenvalues of the matrix
\begin{align}\label{matrix}
A(\xi):=
\begin{bmatrix}
0 & 2 & 0\\
\xi & -\frac{\sigma^2}{2}\xi^2 & 1\\
0 & 2\xi & 0
\end{bmatrix}
\end{align}
associated with \eqref{ODE} are given by
\begin{align}\label{eq:eigs}
\lambda_0=0,\qquad
\lambda_\pm(\xi)=
-\frac{\sigma^2\xi^2}{4}
\pm
\frac14\sqrt{\sigma^4\xi^4+64\xi},
\end{align}
with the square root interpreted as the complex square root when the
argument is negative and the branch with $ \sqrt{x}>0 $ if $ x>0 $ is chosen.
A simple computation shows the very important result that the solutions of the system are uniformly bounded in $ |\xi| $. More precisely, we have the following result.

\begin{lemma}[Uniform bound]\label{lem:lambdaBound}
	For all $\xi\in\R$ and $\sigma\neq 0$,
	\begin{align*}
		\Re \lambda_+(\xi)\le \frac{2}{\sigma^{2/3}}= \lambda_+\left(\frac{2}{\sigma^{4/3}}\right);
	\end{align*}
	moreover,
	\begin{align*}
		\xi\le 0 \quad\mbox{ implies }\quad \Re \lambda_+(\xi)= -\frac{\sigma^2\xi^2}{4}\le 0.
	\end{align*}
\end{lemma}

Notice that an eigenbasis for \eqref{matrix} is
\begin{align*}
v_0(\xi)=\begin{pmatrix}1\\0\\-\xi\end{pmatrix},\qquad
v_\pm(\xi)=\begin{pmatrix}1\\ \frac{\lambda_\pm(\xi)}{2}\\ \xi\end{pmatrix}.
\end{align*}
Moreover, from \eqref{m} and Cauchy-Schwarz inequality we get 
\begin{equation}\label{eq:cone}
|m_2(t;\xi)|\le \sqrt{m_1(t;\xi)m_3(t;\xi)}\le \frac{m_1(t;\xi)+m_3(t;\xi)}{2}, \quad t\in [0,T];\xi\in\mathbb{R},
\end{equation}
thus implying that the second component of the vector $m(t)=(m_1(t;\xi),m_2(t;\xi),m_3(t;\xi))^{\top}$ can be controlled by the other two.\\
We now define the weighted energy
\begin{equation}\label{eq:Fdef}
F(t;\xi):=m_1(t;\xi)+\ang{\xi}^{-1}m_3(t;\xi)\quad\mbox{ with }\quad \ang{\xi}:=\sqrt{1+\xi^2}.
\end{equation}
To ease the notation we also write
\begin{align*}
\gamma:=\frac{\sigma^2}{2}\xi^2\quad\mbox{ and }\quad\Delta:=\sqrt{\gamma^2+16\xi}
\end{align*}
so that
\begin{align*}
\lambda_\pm(\xi)=\frac{-\gamma\pm \Delta}{2}.
\end{align*}
Writing the solution to \eqref{ODE} in eigenvector coordinates, i.e.
\begin{align*}
m(t;\xi)=q_0v_0(\xi)+q_+(t;\xi)v_+(\xi) + q_-(t;\xi)v_-(\xi), 
\end{align*}
we have
\begin{align}\label{eq:q-evol}
q_0(t)=q_0(0),\qquad
q_+(t;\xi)=e^{\lambda_+(\xi) t}q_+(0),\qquad
q_-(t;\xi)=e^{\lambda_-(\xi) t}q_-(0).
\end{align}
Moreover,
\begin{align*}
m_1(t;\xi)=q_0+q_+(t;\xi)+q_-(t;\xi),\qquad
m_3(t;\xi)=-\xi q_0+\xi\bigl(q_+(t;\xi)+q_-(t;\xi)\bigr),
\end{align*}
and therefore using \eqref{eq:Fdef} we get
\begin{align*}
F(t;\xi)=\Bigl(1-\frac{\xi}{\ang{\xi}}\Bigr)q_0
     +\Bigl(1+\frac{\xi}{\ang{\xi}}\Bigr)\bigl(q_+(t;\xi)+q_-(t;\xi)\bigr).
\end{align*}
On the other hand, since $0\le 1\pm \frac{\xi}{\ang{\xi}}\le 2$, we obtain the basic estimate
\begin{equation}\label{eq:F-basic}
F(t;\xi)\le 2|q_0|+2|q_+(t;\xi)|+2|q_-(t;\xi)|.
\end{equation}

\begin{notation*}
In the sequel to ease the notation we will omit to write in the functions $m$, $q_+$, $q_-$, $\lambda_+$ and $\lambda_-$ the explicit dependence on $\xi$.
\end{notation*}
  
Let $m(0):=(m_{10},m_{20},m_{30})^\top$. A direct computation yields
\begin{align}
q_0(0) &= \frac12 m_{10}-\frac{1}{2\eta}m_{30}, \label{eq:q0}\\
q_+(0) &= -\frac{\lambda_-}{2\Delta}m_{10}
         +\frac{2}{\Delta}m_{20}
         -\frac{\lambda_-}{2\eta\Delta}m_{30},\label{eq:qp}\\
q_-(0) &= \frac{\lambda_+}{2\Delta}m_{10}
         -\frac{2}{\Delta}m_{20}
         +\frac{\lambda_+}{2\eta\Delta}m_{30}. \label{eq:qm}
\end{align}
Note that, from \eqref{eq:Fdef} at $t=0$,
\begin{equation}\label{eq:F0-trivial}
m_{10}\le F(0;\xi),\qquad m_{30}\le \ang{\xi}\,F(0;\xi).
\end{equation}
Also, by \eqref{eq:cone} at $t=0$,
\begin{equation}\label{eq:m20-bound}
|m_{20}|\le \frac{m_{10}+m_{30}}{2}.
\end{equation}

\begin{lemma}[Large-$|\xi|$ coefficient bounds]\label{lem:coef}
There exist $\xi_0\ge1$ and $C>0$ (depending only on $\sigma$) such that for all $|\xi|\ge\xi_0$,
\begin{equation}\label{eq:coef}
\frac1{|\Delta|}\le\frac{C}{\xi^2},\qquad
\Bigl|\frac{\lambda_-}{\Delta}\Bigr|\le C,\qquad
\Bigl|\frac{\lambda_-}{\xi\Delta}\Bigr|\le\frac{C}{|\xi|},\qquad
\Bigl|\frac{\lambda_+}{\Delta}\Bigr|\le\frac{C}{|\xi|^3},\qquad
\Bigl|\frac{\lambda_+}{\xi\Delta}\Bigr|\le\frac{C}{|\xi|^4}.
\end{equation}
\end{lemma}

\begin{proof}
Since $\gamma=\frac{\sigma^2}{2}\xi^2$ and $\Delta=\sqrt{\gamma^2+16\xi}$, we have $\Delta\sim \gamma\sim \xi^2$ as $|\xi|\to\infty$,
hence $1/|\Delta|\lesssim 1/\xi^2$.\\
Moreover, $\lambda_-=\frac{-\gamma-\Delta}{2}$ satisfies $|\lambda_-|\sim \gamma\sim \xi^2$, so $|\lambda_-|/|\Delta|\lesssim 1$ and
$|\lambda_-|/(|\xi||\Delta|)\lesssim 1/|\xi|$.\\
Finally, $\lambda_+=\frac{-\gamma+\Delta}{2}=O(1/\xi)$ as $|\xi|\to\infty$, because
$\Delta=\gamma\sqrt{1+16\xi/\gamma^2}=\gamma+O(1/\xi)$, hence
$\lambda_+=(-\gamma+\Delta)/2=O(1/\xi)$. Therefore
$|\lambda_+|/|\Delta|\lesssim (1/|\xi|)/\xi^2=1/|\xi|^3$ and similarly
$|\lambda_+|/(|\xi||\Delta|)\lesssim 1/|\xi|^4$.
\end{proof}

\begin{lemma}[Control by $F(0;\xi)$ of eigenvector coefficients]\label{lem:qF}
There exist $\xi_0\ge1$ and $C>0$ such that for all $|\xi|\ge\xi_0$,
\begin{equation}\label{eq:qF}
|q_0(0)|+|q_+(0)|+|q_-(0)|\le C\,F(0;\xi).
\end{equation}
\end{lemma}

\begin{proof}
\emph{Step 1: bound $q_0(0)$.}
By \eqref{eq:q0} and \eqref{eq:F0-trivial},
\[
|q_0(0)|\le \tfrac12 m_{10}+\tfrac1{2|\xi|}m_{30}
\le \tfrac12F(0;\xi)+\tfrac{\ang{\xi}}{2|\xi|}\frac{m_{30}}{\ang{\xi}}
\le C F(0;\xi),
\]
since $\ang{\xi}/|\xi|\le \sqrt{2}$ for $|\xi|\ge 1$.

\medskip
\noindent
\emph{Step 2: bound $q_+(0)$.}
Using \eqref{eq:qp} and Lemma~\ref{lem:coef},
\[
|q_+(0)|\le C m_{10}+\frac{C}{\xi^2}|m_{20}|+\frac{C}{|\xi|}m_{30}.
\]
We bound each term by $F(0;\xi)$. First, $m_{10}\le F(0;\xi)$.
Next, by \eqref{eq:m20-bound} and \eqref{eq:F0-trivial},
\[
\frac{1}{\xi^2}|m_{20}|\le \frac{1}{2\xi^2}(m_{10}+m_{30})
\le \frac{C}{\xi^2}F(0;\xi)+\frac{C}{\xi^2}\ang{\xi}F(0;\xi)
\le \frac{C}{|\xi|}F(0;\xi)\le C F(0;\xi),
\]
for $|\xi|\ge 1$.
Finally,
\[
\frac{1}{|\xi|}m_{30}\le \frac{\ang{\xi}}{|\xi|}\frac{m_{30}}{\ang{\xi}}
\le C F(0;\xi).
\]
Hence $|q_+(0)|\le C F(0;\xi)$.

\medskip
\noindent
\emph{Step 3: bound $q_-(0)$.}
Using \eqref{eq:qm} and Lemma~\ref{lem:coef},
\[
|q_-(0)|\le \frac{C}{|\xi|^3}m_{10}+\frac{C}{\xi^2}|m_{20}|+\frac{C}{|\xi|^4}m_{30}
\le C F(0;\xi),
\]
using again \eqref{eq:m20-bound} and \eqref{eq:F0-trivial} and $|\xi|\ge 1$. Summing the three bounds yields \eqref{eq:qF}.
\end{proof}

\begin{theorem}[Uniform weighted-energy bound for large $|\xi|$]\label{thm:F}
Fix $T>0$. There exist constants $C_1,C_2>0$ and $\xi_0\ge1$ such that for all $|\xi|\ge\xi_0$
and all $t\in[0,T]$,
\begin{equation}\label{eq:F-final}
F(t;\xi)\le C_1 e^{C_2 t}F(0;\xi).
\end{equation}
\end{theorem}

\begin{proof}
From \eqref{eq:F-basic} and \eqref{eq:q-evol},
\[
F(t;\xi)\le 2|q_0(0)|+2e^{\Re\lambda_+ t}|q_+(0)|+2e^{\Re\lambda_- t}|q_-(0)|.
\]
For $|\xi|$ large, $\Re\lambda_+=O(1/|\xi|)$ and $\Re\lambda_- \sim -\gamma \le 0$,
so there exist $\xi_0$ and $C_2>0$ such that for all $|\xi|\ge\xi_0$,
\[
e^{\Re\lambda_+ t}\le e^{C_2 t},\qquad e^{\Re\lambda_- t}\le 1,\qquad t\in[0,T].
\]
Therefore,
\[
F(t;\xi)\le 2e^{C_2 t}\bigl(|q_0(0)|+|q_+(0)|+|q_-(0)|\bigr).
\]
Applying Lemma~\ref{lem:qF} gives \eqref{eq:F-final}.
\end{proof}

The large-$|\xi|$ estimate in Theorem \ref{thm:F} leaves open the compact region
$|\xi|\le \xi_0$. On this region we can bound the weighted energy by a
uniform constant using compactness.

\begin{lemma}[Bounded-$\xi$ patch]\label{lem:boundedPatch}
Fix $T>0$ and $\xi_0\ge 1$. There exists a constant $M=M(T;\xi_0,\sigma)>0$ such that
for all $|\xi|\le \xi_0$ and all $t\in[0,T]$,
\[
F(t;\xi)\le M\,F(0;\xi).
\]
Consequently, for any $C_2\ge 0$,
\[
F(t;\xi)\le M\,e^{C_2 t}\,F(0;\xi),\qquad t\in[0,T].
\]
\end{lemma}

\begin{proof}
Recalling the definition of $A(\xi)$ in \eqref{matrix}, define the diagonal weight
\begin{align*}
W(\xi):=\mathrm{diag }\left[1,\,1,\,\ang{\xi}^{-1}\right];
\end{align*}
then for each $\xi$,
\begin{align*}
W(\xi)m(t)=W(\xi)e^{tA(\xi)}W(\xi)^{-1}\,W(\xi)m(0).
\end{align*}
The map $(t;\xi)\mapsto B(t;\xi):=W(\xi)e^{tA(\xi)}W(\xi)^{-1}$ is continuous on the compact set
$[0,T]\times[-\xi_0;\xi_0]$ (matrix exponential is continuous and $W(\xi)$ is smooth); hence
\begin{align*}
M:=\sup_{t\in[0,T]}\sup_{|\xi|\le \xi_0}\|B(t;\xi)\|_{1\to 1}<\infty.
\end{align*}
Moreover, using that $m_1,m_3\ge 0$, the definition \eqref{eq:Fdef} implies
\[
F(t;\xi)=m_1(t)+\ang{\xi}^{-1}m_3(t)=\|W(\xi)m(t)\|_{\ell^1}-|m_2(t)|\le \|W(\xi)m(t)\|_{\ell^1},
\]
and similarly $F(0;\xi)\ge 0$ and $\|W(\xi)m(0)\|_{\ell^1}=F(0)+|m_2(0)|\le 2F(0;\xi)$ by \eqref{eq:cone}.
Therefore,
\[
F(t;\xi)\le \|W(\xi)m(t)\|_{\ell^1}
\le \|B(t;\xi)\|_{1\to 1}\,\|W(\xi)m(0)\|_{\ell^1}
\le 2M\,F(0;\xi),
\]
uniformly for $t\in[0,T]$ and $|\xi|\le \xi_0$.
The last statement follows from $e^{C_2 t}\ge 1$.
\end{proof}

Combining Lemma \ref{lem:boundedPatch} with Theorem \ref{thm:F} yields constants
$C_1^\ast,C_2^\ast>0$ such that for all $\xi\in\R$ and all $t\in[0,T]$,
\[
F(t;\xi)\le C_1^\ast e^{C_2^\ast t}F(0;\xi).
\]
Thus we have proved:

\begin{theorem}\label{thm:Global}
Fix $T>0$. There exist constants $C_1,C_2>0$ such that
for all $\xi \in \R$
and all $t\in[0,T]$,
\begin{align}\label{eq:F-finalGlobal}
F(t;\xi)\le C_1 e^{C_2 t}F(0;\xi).
\end{align}
\end{theorem}

We now apply the standard Plancherel technique that upgrades the frequency estimate for $F(t;\xi)$
to Sobolev moment bounds in the $x$ variable.

\begin{lemma}[Sobolev moment bounds from the weighted energy]\label{lem:sobolevMoment2}
Fix $s\in\R$ and $T>0$. Assume that for all $\xi\in\R$ and all $t\in[0,T]$,
\begin{align}\label{GES}
F(t;\xi)\le C_1 e^{C_2 t}F(0;\xi),
\end{align}
where
\begin{align*}
F(t;\xi)=m_1(t;\xi)+\ang{\xi}^{-1}m_3(t;\xi),\qquad
m_1(t;\xi)=\E\,|\widehat U(t;\xi)|^2,\quad
m_3(t;\xi)=\E\,|\widehat V(t;\xi)|^2.
\end{align*}
Then,
\begin{align*}
\sup_{t\in[0,T]}
\left[\E\|U(t)\|_{H_x^s}^2+\E\|V(t)\|_{H_x^{s-\frac12}}^2\right]
\le C_1 e^{C_2 T}\left[
\|\varphi_0\|_{H_x^s}^2+\|\varphi_1\|_{H_x^{s-\frac12}}^2
\right].
\end{align*}
\end{lemma}

\begin{proof}
Multiply the frequency estimate (\ref{GES}), which was proved in
Theorem \ref{eq:F-finalGlobal}, by $\ang{\xi}^{2s}$ and integrate over $\xi\in\R$:
\begin{align*}
\int_{\R}\ang{\xi}^{2s}F(t;\xi)\,d\xi
\le
C_1 e^{C_2 t}\int_{\R}\ang{\xi}^{2s}F(0;\xi)\,d\xi.
\end{align*}
By definition of $F$,
\begin{align*}
\int_{\R}\ang{\xi}^{2s}F(t;\xi)\,d\xi
=
\int_{\R}\ang{\xi}^{2s}\E|\widehat U(t;\xi)|^2d\xi
+
\int_{\R}\ang{\xi}^{2s-1}\E|\widehat V(t;\xi)|^2d\xi,
\end{align*}
and using Fubini's theorem and Plancherel's identity in the $x$-variable gives
\begin{align*}
\int_{\R}\ang{\xi}^{2s}\E|\widehat U(t;\xi)|^2d\xi
=\E\|U(t)\|_{H_x^s}^2,
\qquad
\int_{\R}\ang{\xi}^{2s-1}\E|\widehat V(t;\xi)|^2\,d\xi
=\E\|V(t)\|_{H_x^{s-\frac12}}^2,
\end{align*}
since $\ang{\xi}^{2s-1}=\ang{\xi}^{2(s-\frac12)}$. The same identities hold at $t=0$ with
$U(0)=\varphi_0$ and $V(0)=\varphi_1$. Taking the supremum over $t\in[0,T]$ yields the claim.
\end{proof}

An immediate consequence of the  Lemma \ref{lem:sobolevMoment2} and Theorem \ref{thm:Global} is:

\begin{theorem}[Mean-square $H_x^s\times H_x^{s-\frac12}$ well-posedness]
\label{thm:meansquare}
Assume that the estimate of Lemma~\ref{lem:sobolevMoment2} holds for a given $s\in\R$.
Then the homogeneous Cauchy problem \eqref{eq:STRA} is well-posed in
$H_x^s\times H_x^{s-\frac12}$ in the \emph{mean-square sense}, namely:
\begin{itemize}
\item for every $(\varphi_0,\varphi_1)\in H_x^s\times H_x^{s-\frac12}$ there exists a unique
adapted solution $(U,V)$;
\item the solution satisfies
\begin{align*}
U\in L^\infty([0,T];L^2(\Omega;H_x^s)),\qquad
V\in L^\infty([0,T];L^2(\Omega;H_x^{s-\frac12}));
\end{align*}
\item the solution map $(\varphi_0,\varphi_1)\mapsto (U,V)$ is continuous with respect to
these norms.
\end{itemize}
\end{theorem}

\begin{corollary}[Mean-square $C^\infty$ well-posedness]
Assume that the estimate of Lemma \ref{lem:sobolevMoment2} holds for all $s\ge0$.
Then the Cauchy problem (\ref{eq:STRA}) is well-posed in $H_x^\infty$ in the mean-square
sense:
\begin{align*}
U\in \bigcap_{k\ge0} L^\infty([0,T];L^2(\Omega;H_x^k)),
\qquad
V\in \bigcap_{k\ge0} L^\infty([0,T];L^2(\Omega;H_x^{k-\frac12})).
\end{align*}
\end{corollary}

Due to the fact that (\ref{eq:MSF}) is a linear system we can also
improve the regularity of the solution process from $ L^{\infty}[0,T]  $ to
$ C[0,T] $ utilizing the following proposition:

\begin{proposition}[Time continuity in $L^2(\Omega;H_x^s)$ from mode-wise SDE continuity]\label{lem:timeCont}
Fix $s\in\R$ and $T>0$. Let $(U,V)$ be a (homogeneous) solution whose $x$-Fourier
transform $\widehat U(t;\xi)$ and $\widehat{V}(t;\xi)$ satisfy,
for each fixed $\xi\in\R$, a linear finite-dimensional SDE system with continuous
sample paths in $t$ (e.g.\ the mode system obtained by Fourier transforming a
linear stochastic partial differential equation in $x$). Assume moreover the uniform Sobolev moment bounds
\begin{align}\label{eq:uniformSob}
\sup_{t\in[0,T]}\E\|U(t)\|_{H_x^s}^2<\infty,
\qquad
\sup_{t\in[0,T]}\E\|V(t)\|_{H_x^{s-\frac12}}^2<\infty.
\end{align}
Then,
\begin{align*}
U\in C\big([0,T];L^2(\Omega;H_x^s)\big),
\qquad
V\in C\big([0,T];L^2(\Omega;H_x^{s-\frac12})\big).
\end{align*}
\end{proposition}

\begin{proof}
We prove the statement for $U$; the proof for $V$ is identical. Fix $t_0\in[0,T]$ and let $t$ tend to $t_0$. By Plancherel's identity in $x$,
\begin{equation}\label{eq:planch}
\E\|U(t)-U(t_0)\|_{H_x^s}^2
=
\int_{\R}\ang{\xi}^{2s}\E\big|\widehat U(t;\xi)-\widehat U(t_0;\xi)\big|^2d\xi.
\end{equation}
For each fixed $\xi$, the assumed continuity of the mode process yields
\begin{align*}
\E\big|\widehat U(t;\xi)-\widehat U(t_0;\xi)\big|^2\longrightarrow 0,
\quad\mbox{ as } t\to t_0.
\end{align*}
Moreover,
\begin{align*}
\big|\widehat U(t;\xi)-\widehat U(t_0;\xi)\big|^2
\le 2|\widehat U(t;\xi)|^2+2|\widehat U(t_0;\xi)|^2,
\end{align*}
hence
\begin{align*}
\ang{\xi}^{2s}\E\big|\widehat U(t;\xi)-\widehat U(t_0;\xi)\big|^2
\le
4\sup_{\tau\in[0,T]}\ang{\xi}^{2s}\E|\widehat U(\tau;\xi)|^2.
\end{align*}
The right-hand side is integrable in $\xi$ because \eqref{eq:uniformSob} is
equivalent by Plancherel to
\begin{align*}
\sup_{\tau\in[0,T]}\int_{\R}\ang{\xi}^{2s}\E|\widehat U(\tau;\xi)|^2d\xi<\infty.
\end{align*}
Therefore dominated convergence applies in \eqref{eq:planch} and yields
\begin{align*}
	\E\|U(t)-U(t_0)\|_{H_x^s}^2\longrightarrow 0
\quad\mbox{ as } t\to t_0,
\end{align*}
which is exactly $U\in C([0,T];L^2(\Omega;H_x^s))$.
\end{proof}

Therefore finally we have:

\begin{corollary}[Continuity $C^\infty$ well-posedness]\label{finRes}
Assume that the estimate of Lemma \ref{lem:sobolevMoment2} holds for all $s\ge0$.
Then the Cauchy problem (\ref{eq:STRA}) is well-posed in $H_x^\infty$ in the continuous 
sense:
\begin{align*}
U\in \bigcap_{k\ge0}C\big([0,T];L^2(\Omega;H_x^k)\big) ,
\qquad
V\in \bigcap_{k\ge0} C\big([0,T];L^2(\Omega;H_x^{k-\frac12})\big)
\end{align*}
\end{corollary}



  



\bigskip
\bigskip

\noindent\textbf{Author Contributions}: All authors contributed equally.

\noindent\textbf{Data availability}: No datasets were generated or analysed during the current study.

\noindent\textbf{Conflict of interest}: The authors declare no competing interests.\\

\bibliographystyle{plain}
\bibliography{whsn_2}

\end{document}